\documentclass[12pt]{amsart}
\usepackage{amssymb,amsmath,amscd,mathtools}

\theoremstyle{plain}
\newtheorem{theorem}{Theorem}[section]

\newtheorem*{quest}{Main Question}
\newtheorem{corollary}[theorem]{Corollary}

\theoremstyle{definition}

\numberwithin{equation}{section}

\newcommand{\ipg}[2]{\langle #1,#2 \rangle_\mathbb{G}}

\newcommand{\tp}{\texttt{p}}
\renewcommand{\span}{\operatorname{span}}
\newcommand{\divergence}{\operatorname{div}}
\def\abs#1{\left|#1\right|}
\def\norm#1{\left\|#1\right\|}
\newcommand{\defeq}{\vcentcolon=}

\begin{document}

\title[Generalizations of the Drift Laplace Equation]{Generalizations of the Drift Laplace Equation over the Quaternions in a Class of Grushin-Type Spaces}

\author{Thomas Bieske}
\author{Keller L. Blackwell}
\address{Department of Mathematics\\
University of South Florida\\ 
Tampa, FL 33620, USA}
\email{tbieske@mail.usf.edu}
\address{Department of Mathematics\\
University of South Florida\\ 
Tampa, FL 33620, USA}
\email{kellerb@mail.usf.edu}

\date{June 10,2019}

\begin{abstract}
Beals, Gaveau, and Greiner [1996] establish a formula for the fundamental solution to the Laplace equation with drift term in Grushin-type planes.  The first author and Childers [2013] expanded these results by invoking a p-Laplace-type generalization that encompasses these formulas while the authors [2019] explored a different natural generalization of the p-Laplace equation with drift term that also encompasses these formulas. In both, the drift term lies in the complex domain. We extend these results by considering a drift term in the quaternion realm and 
show our solutions are stable under limits as p tends to infinity.  
\end{abstract}

\maketitle
\section{Introduction and Motivation}
In \cite{BGG}, Beals, Gaveau, and Greiner establish a formula for the fundamental solution to the Laplace equation with drift term in a large class of sub-Riemannian spaces, which includes the so-called Grushin-type planes. In \cite{B:C}, the first author and Childers expanded these results by invoking a $\tp$-Laplace-type generalization that encompasses the formulas of \cite{BGG} while in \cite{BB}, the authors explored a different natural generalization of the $\tp$-Laplace equation with drift term that also encompasses the formulas of \cite{BGG}. In both cases, the drift term lies in the complex domain. In this paper, we will consider both approaches, but with a drift term in the quaternion realm and create an extension of both cases. We will then show our solutions are stable under limits when $\tp\to\infty$.

This paper is the result of an undergraduate research project by the second author under the direction of the first. The second author would like to thank the University of South Florida Honors College and the Department of Mathematics and Statistics for their support and research opportunities.
\section{Grushin-type planes}
We begin with a brief discussion of our environment. The Grushin-type planes are a class of sub-Riemannian spaces lacking an algebraic group law. We begin with $\mathbb{R}^{2}$, possessing coordinates $(y_1, y_2)$, $a\in \mathbb{R}$, $c\in \mathbb{R}\setminus\{0\}$ and $n\in \mathbb{N}$.  We use them to construct the vector fields
\begin{equation*}
Y_1  =  \frac{\partial}{\partial y_1} \ \textmd{and}\ \
Y_2  =   c(y_1-a)^n\frac{\partial}{\partial y_2}.
\end{equation*}
For these vector fields, the only (possibly) nonzero Lie bracket is
\begin{equation*}
[Y_1,Y_2]=cn(y_1-a)^{n-1}\frac{\partial}{\partial y_2}.
\end{equation*}
Because $n\in \mathbb{N}$, it follows that H\"{o}rmander's condition (see, for example, \cite{BR:SRG}) is satisfied by these vector fields.

We will put a (singular) inner product on $\mathbb{R}^{2}$, denoted $\ipg{\cdot}{\cdot}$,  with related norm $\|\cdot\|_{\mathbb{G}}$, so that the collection $\{Y_{1}, Y_{2}\}$ 
forms an orthonormal basis. We then have a sub-Riemannian space that we will call $g_{n}$, which is also the tangent space to a generalized Grushin-type plane $\mathbb{G}_n$. Points in $\mathbb{G}_n$ will also be denoted by
$p=(y_{1}, y_{2})$.  The Carnot-Carath\'{e}odory distance on $\mathbb{G}_n$ is defined for points $p$ and $q$ as follows:
\begin{eqnarray*}
d_{\mathbb{G}}(p,q)=\inf_{\Gamma}\int \|\gamma'(t)\|_{\mathbb{G}}\;dt
\end{eqnarray*} 
with $\Gamma$ the set of curves $\gamma$ such that $\gamma(0)=p$, $\gamma(1)=q$ and
$
\gamma'(t)\in \span\{Y_{1}(\gamma(t)),Y_{2}(\gamma(t))\}.
$
By Chow's theorem, this is an honest metric.

We shall now discuss calculus on the Grushin-type planes. Given a smooth function $f$ on $\mathbb{G}_n$, we define the horizontal gradient of $f$ as
\begin{equation*}
 \nabla_{0} f(p) = \big(Y_1f(p),Y_2f(p)\big).
\end{equation*}

Using these derivatives, we consider a key operator on $C^2_{\mathbb{G}}$ functions, namely the \tp-Laplacian for $1<\tp<\infty$, given by
\begin{eqnarray}
\Delta_\tp f & = & \divergence(\|\nabla_{0} f\|_{\mathbb{G}}^{\tp-2}\nabla_{0}f)   =  Y_1\big(\|\nabla_{0} f\|_{\mathbb{G}}^{\tp-2}Y_1f\big)+Y_2\big(\|\nabla_{0} f\|_{\mathbb{G}}^{\tp-2}Y_2f\big)  \nonumber\\
& = & \frac{1}{2}(\tp-2)\|\nabla_{0} f\|_{\mathbb{G}}^{\tp-4}\big(Y_1\|\nabla_{0} f\|_{\mathbb{G}}^{2}Y_1f+Y_2\|\nabla_{0} f\|_{\mathbb{G}}^{2}Y_2f\big)  \label{reductiong}\\
& & \mbox{} +\|\nabla_{0} f\|_{\mathbb{G}}^{\tp-2}\big(Y_1Y_1f+Y_2Y_2f\big).  \nonumber
\end{eqnarray}

\section{Motivating Results}
\subsection{Grushin-type Planes}
The first author and Gong \cite{BG} proved the following in the Grushin-type planes.
\begin{theorem}[\cite{BG}]\label{T6}
Let $1<\tp<\infty$ and define $$f(y_1,y_2)=c^2(y_1-a)^{(2n+2)}+(n+1)^2 (y_2-b)^2.$$  
For $\tp \neq n+2$, consider $$\tau_{\tp}=\frac{n+2-\tp}{(2n+2)(1-\tp)}$$
so that in $\mathbb{G}_n\setminus\{(a,b)\}$  we have the well-defined function
\begin{eqnarray*}
\psi_{\tp}=\left\{\begin{array}{cc}
f(y_1,y_2)^{\tau_{\tp}} & \tp \neq n+2 \\
\log f(y_1,y_2) & \tp = n+2.
\end{array}\right.
\end{eqnarray*}
Then, $\Delta_{\tp}\psi_{\tp}=0$ in $\mathbb{G}_n\setminus\{(a,b)\}$.
\end{theorem} 
In the Grushin-type planes, Beals, Gaveau and Greiner \cite{BGG} extend this equation as shown in the following theorem. 
\begin{theorem}[\cite{BGG}]\label{MG}
Let $L \in \mathbb{R}$. Consider the following quantities,
\begin{equation*}
\alpha = \frac{-n}{(2n+2)}(1+L) \ \ \textmd{and}\ \ \beta  =  \frac{-n}{(2n+2)}(1-L).
\end{equation*}
We use these constants with the functions
\begin{eqnarray*}
g(y_1,y_2) & = & c(y_1-a)^{n+1}+i(n+1)(y_2-b)\\ 
h(y_1,y_2) & = & c(y_1-a)^{n+1}-i(n+1)(y_2-b)
\end{eqnarray*}
to define our main function $f(y_1,y_2)$, given by
\begin{eqnarray*}
f(y_1,y_2) & = & g(y_1,y_2)^{\alpha}h(y_1,y_2)^{\beta}.
\end{eqnarray*}
Then, $\mathcal{D}(f) \defeq \Delta_{2}f+iL[Y_1,Y_2]f=0$ in $\mathbb{G}_n\setminus\{(a,b)\}$.
\end{theorem}

Non-linear generalizations of Theorem \ref{MG} have been explored by the first author and Childers in \cite{B:C} and by the authors in \cite{BB}. The following theorem  extends Theorem \ref{MG} through a $\tp$-Laplace type divergence form.

\begin{theorem}[\cite{B:C}]\label{childersdiv}
For $L \in \mathbb{R}$ with $L \ne \pm 1$, consider the following parameters for $\tp \ne n+2$:
\begin{equation*}
    \alpha = \frac{n+2-\tp}{(1-\tp)(2n+2)}(1+L) \quad\text{and}\quad \beta = \frac{n+2-\tp}{(1-\tp)(2n+2)}(1-L)
\end{equation*}
with the functions:
\begin{eqnarray*}
g(y_1, y_2) &=& c(y_1 - a)^{n+1} + i(n+1)(y_2-b)\\
h(y_1, y_2) &=& c(y_1 - a)^{n+1} - i(n+1)(y_2-b)
\end{eqnarray*}
to define the main function:
\begin{equation*}
    f_{\tp, L} = \left\{\begin{array}{cc}
g(y_1,y_2)^{\alpha} h(y_1,y_2)^{\beta} & \tp \ne n+2 \\
\log\big(g(y_1, y_2)^{1+L} h(y_1, y_2)^{1-L}\big) & \tp = n+2.
\end{array}\right.
\end{equation*}
Then
\begin{equation*}
    \overline{\Delta_\tp} f_{\tp, L} \defeq \divergence\left( \begin{Vmatrix}
Y_1 f_{\tp, L} + iL Y_2 f_{\tp, L}\\
Y_2 f_{\tp, L} - iL Y_1 f_{\tp, L}
\end{Vmatrix}^{\tp-2}_\mathbb{G} \begin{pmatrix}
Y_1 f_{\tp, L} + iL Y_2 f_{\tp, L}\\
Y_2 f_{\tp, L} - iL Y_1 f_{\tp, L}
\end{pmatrix}  \right) = 0.
\end{equation*}
\end{theorem}

The following theorem of the authors takes an alternative approach to extending Theorem \ref{MG} through a generalization of the drift term.

\begin{theorem}[\cite{BB}]\label{BBpdrift}
For $L \in \mathbb{R}$ with:
\begin{equation*}
    L \ne \pm \frac{n+2-\tp}{n(1-\tp)}
\end{equation*}
consider the parameters:
\begin{eqnarray*}
    \alpha = \frac{n+2-\tp - Ln(1- \tp)}{2(n+1)(1-\tp)} & \textmd{and} & \beta = \frac{ n+2-\tp + Ln(1- \tp)}{2(n+1)(1-\tp)}
\end{eqnarray*}
with the functions
\begin{eqnarray*}
g(y_1,y_2) & = & c(y_1-a)^{n+1}+i(n+1)(y_2-b)\\ 
h(y_1,y_2) & = & c(y_1-a)^{n+1}-i(n+1)(y_2-b)
\end{eqnarray*}
to define the main function:
\begin{eqnarray}
 f_{\tp, L}(y_1, y_2) = g(y_1, y_2)^\alpha h(y_1, y_2)^\beta \:.
\end{eqnarray}
Then on $\mathbb{G}_n \setminus \lbrace (a,b) \rbrace$, we have:
\begin{eqnarray*}
\mathcal{G}_{\tp, L} \left( f_{\tp, L} \right) \defeq \Delta_{\tp} f_{\tp, L} + iL \left[ Y_1, Y_2 \right] \left( \norm{\nabla_0 f_{\tp, L}}^{\tp -2}_\mathbb{G} f_{\tp, L}  \right) = 0.
\end{eqnarray*}
\end{theorem}

\begin{quest}
We wish to extend the preceding generalizations of Theorem \ref{MG} over the quaternions, denoted $\mathbb{H}$. Recall that the solved partial differential equation of Theorem \ref{MG},
\begin{equation*}
    \Delta_2 f + iL[Y_1, Y_2] f = 0,
\end{equation*}
features a drift term bearing the purely complex-imaginary coefficient $iL \in \mathbb{C}$. We ask if this coefficient can be generalized to a purely quaternion-imaginary coefficient of the form:
\begin{equation*}
    Q = Li + Mj + Nk \in \mathbb{H} \setminus \mathbb{R}  
\end{equation*}
where the case of $Q = 0$ reduces to the result of Theorem \ref{T6}. With respect to Theorem \ref{childersdiv}, we explore smooth solutions to the generalization:
\begin{equation*}
    \overline{\Delta_\tp} f \defeq \divergence\left( \begin{Vmatrix}
Y_1 f + Q Y_2 f\\
Y_2 f - Q Y_1 f
\end{Vmatrix}^{\tp-2}_\mathbb{G} \begin{pmatrix}
Y_1 f + Q Y_2 f\\
Y_2 f - Q Y_1 f
\end{pmatrix}  \right) = 0.
\end{equation*}
With respect to Theorem \ref{BBpdrift}, we explore smooth solutions to the generalization:
\begin{eqnarray*}
\mathcal{G}_{\tp, Q} \left( f \right) \defeq \Delta_{\tp} f + Q \left[ Y_1, Y_2 \right] \left( \norm{\nabla_0 f}^{\tp -2}_\mathbb{G} f  \right) = 0.
\end{eqnarray*}
\end{quest}

\section{A $\tp$-Laplacian Type Generalization over $\mathbb{H}$}
\subsection{Case I: $\mathbf{L + M + N \ne 0}$} \mbox{}\\

Let $Q = Li + Mj +Nk \in \mathbb{H} \setminus\mathbb{R}$ with $L + M + N \ne 0$. We consider the following parameters:
\begin{eqnarray*}
	\mu &=& \frac{\sqrt{\abs{Q^2}}}{\abs{L + M + N}}\\
	\omega &=& \frac{Q}{L + M + N}\\
    \xi &=& \sqrt{\abs{Q^2}} (L + M + N)\\
    \alpha &=& \frac{n+2-\tp}{(1-\tp)(2n+2)} (1 + \xi)\\
\textmd{and\ }    \beta &=& \frac{n+2-\tp}{(1-\tp)(2n+2)} (1-\xi)
\end{eqnarray*}
where $\xi \ne \pm 1$. We use these constants with the functions:
\begin{eqnarray*}
	 g(y_1, y_2) &=& \mu c(y_1-a)^{n+1} + \omega (n+1)(y_2 - b)\\
    h(y_1, y_2) &=& \mu c(y_1-a)^{n+1} - \omega (n+1)(y_2 - b)
\end{eqnarray*}
to define our main function:
\begin{equation}\label{ChildersCoreA}
f_{\tp, Q} (y_1, y_2) = \left\{\begin{array}{cc}
g(y_1,y_2)^{\alpha} h(y_1,y_2)^{\beta} & \tp \ne n+2 \\
\log\big(g(y_1, y_2)^{1+\xi} h(y_1, y_2)^{1-\xi}\big) & \tp = n+2.
\end{array}\right.
\end{equation}
Using Equation \ref{ChildersCoreA}, we have the following theorem.

\begin{theorem}\label{ChildersClaimA}
Let $Q = Li + Mj +Nk \in \mathbb{H} \setminus \mathbb{R}$ with $L + M + N \ne 0$. On $G_n \setminus \lbrace (a,b) \rbrace$, \\ we have:
\begin{equation*}
\overline{\Delta_\tp} f_{\tp, Q} \defeq \divergence_\mathbb{G} \left( \begin{Vmatrix}
Y_1 f_{\tp, Q} + Q Y_2 f_{\tp, Q}\\
Y_2 f_{\tp, Q} - Q Y_1 f_{\tp, Q}
\end{Vmatrix}^{\tp-2}_\mathbb{G} \begin{pmatrix}
Y_1 f_{\tp, Q} + Q Y_2 f_{\tp, Q}\\
Y_2 f_{\tp, Q} - Q Y_1 f_{\tp, Q}
\end{pmatrix}  \right) = 0.
\end{equation*}
\end{theorem}

\begin{proof}
Suppressing arguments and subscripts, we let:
\begin{equation*}
\Upsilon \defeq \begin{pmatrix}
\Upsilon_1\\
\Upsilon_2
\end{pmatrix} = \begin{pmatrix}
Y_1 f + Q Y_2 f\\
Y_2 f - Q Y_1 f
\end{pmatrix}.
\end{equation*}
Observing that:
\begin{eqnarray*}
\overline{\Delta_\tp} f &=& \divergence\left( \|\Upsilon\|^{\tp-2} \Upsilon \right)\\
&=& \norm{\Upsilon}^{\tp-4} \left( \frac{\tp-2}{2} \sum_{s=1}^2 Y_s \norm{\Upsilon}^2 \Upsilon_s + \norm{\Upsilon}^2 (Y_1 \Upsilon_1 + Y_2 \Upsilon_2) \right)
\end{eqnarray*}
it suffices to show:
\begin{equation*}
\Lambda \defeq \frac{\tp-2}{2} \sum_{s=1}^2 Y_s \norm{\Upsilon}^2 \Upsilon_s + \norm{\Upsilon}^2 (Y_1 \Upsilon_1 + Y_2 \Upsilon_2) = 0.
\end{equation*}
For $\tp \ne n+2$, we compute the following:
\begin{eqnarray*}
Y_1 f &=& \mu c(n+1)(y_1 - a)^n g^{\alpha -1} h^{\beta -1} (\alpha h + \beta g )\\
Y_2 f &=& \omega c(n+1)(y_1-a)^n g^{\alpha -1} h^{\beta -1} \left( \alpha h - \beta g \right)\\
Y_1 f + Q Y_2 f &=& \mu c(n+1) (y_1 - a)^n g^{\alpha -1} h^{\beta -1} \left( \alpha h(1 - \xi) + \beta g (1 + \xi) \right)\\
Y_2 f - Q Y_1 f &=& \omega c(n+1)(y_1-a)^n g^{\alpha -1} h^{\beta -1} \left( \alpha h (1 - \xi) - \beta g (1 + \xi)  \right)\\
\textmd{and\ }\norm{\Upsilon}^2 &=& 2 \mu^2 c^2 (n+1)^2 (y_1 - a)^{2n} g^{\alpha + \beta -1} h^{\alpha + \beta -1}  \left( \alpha^2 (1-\xi)^2 + \beta^2(1 + \xi)^2  \right).
\end{eqnarray*}
We then calculate:
\begin{eqnarray*}
Y_1 \Upsilon_1 + Y_2 \Upsilon_2 &=&  \frac{1}{(-1+\tp)^2 gh} \mu^2 c^2 (- 1 + \xi^2) ( 1+n ) (2 + n - \tp) (-2 + \tp) (y_1-a)^{2n} g^\alpha h^\beta\\
Y_1 \norm{\Upsilon}^2 &=& -\frac{1}{(-1+\tp)^3 gh}  \Big( 2 \mu^2 c^2 (1- \xi^2)^2 (n+1) (n+2-\tp)^2 (y_1-a)^{2n-1}    \times  \\
        && \quad g^{\alpha + \beta -1} h^{\alpha + \beta -1}\left( \mu^2 c^2   (y_1 - a)^{2n+2} - \mu^2 n (n+1) (-1+\tp) (y_2 -b)^2 \right) \Big)\\
\textmd{and\ }Y_2 \norm{\Upsilon}^2 &=& \frac{1}{(-1+\tp)^3 gh}  2 \mu^4 c^3 (1-\xi^2)^2  (n+1) (n+2-\tp)^2 (1+n\tp)  \times \\
&&\quad  (y_1 - a)^{3n} (b-y_2) g^{\alpha + \beta -1} h^{\alpha + \beta-1}. 
\end{eqnarray*}
Using the above quantities we compute:
\begin{eqnarray}
\frac{\tp-2}{2}\sum_{s=1}^2 Y_s \norm{\Upsilon}^2 \Upsilon_s &=& -\frac{1}{(-1+\tp)^4 }  \mu^4 c^4 (-1+ \xi^2)^3 (n+1) (n+2-\tp)^3 \times \label{gpart1}\\ \nonumber
&&\quad (y_1-a)^{4n} g^{2\alpha + \beta-2} h^{\alpha + 2\beta-2}(\tp-2)\\ \nonumber
\textmd{and\ } \norm{\Upsilon}^2 (Y_1 \Upsilon_1 + Y_2 \Upsilon_2) &=& \frac{1}{(-1+\tp)^4 } \mu^4 c^4 (n+1) (y_1 - a)^{4n} g^{2\alpha + \beta -2} h^{\alpha + 2\beta -2} \times\\ \nonumber
&&\quad (n+2-\tp)^3 (-1+\xi^2)^3 (\tp-2)
\end{eqnarray}
whereby it follows that $\Lambda = 0$, as desired. The case $\tp = n+2$ is similar and omitted.
\end{proof}

\subsection{Case II: $\mathbf{L + M + N = 0}$}\mbox{}\\

Let $Q = Li + Mj +Nk \in \mathbb{H} \setminus  \mathbb{R} $ with $L + M + N = 0$. We consider the following parameters:
\begin{eqnarray*}
    \xi &=& \sqrt{2\abs{LM + LN + MN}}\\
    \alpha &=& \frac{n+2-\tp}{(1-\tp)(2n+2)} (1 + \xi)\\
 \textmd{and\ }   \beta &=& \frac{n+2-\tp}{(1-\tp)(2n+2)} (1-\xi)
\end{eqnarray*}
where $\xi \ne \pm 1$. We use these constants with the functions:
\begin{eqnarray*}
	 g(y_1, y_2) &=& \xi c(y_1-a)^{n+1} + Q (n+1)(y_2 - b)\\
    h(y_1, y_2) &=& \xi c(y_1-a)^{n+1} - Q (n+1)(y_2 - b)
\end{eqnarray*}
to define our main function:
\begin{equation}\label{ChildersCoreB}
f_{\tp, Q} (y_1, y_2) = \left\{\begin{array}{cc}
g(y_1,y_2)^{\alpha} h(y_1,y_2)^{\beta} & \tp \ne n+2 \\
\log\big(g(y_1, y_2)^{1+\xi} h(y_1, y_2)^{1-\xi}\big) & \tp = n+2.
\end{array}\right.
\end{equation}
Using Equation \ref{ChildersCoreB}, we have the following theorem.
\begin{theorem}\label{ChildersClaimB}
Let $Q = Li + Mj +Nk \in \mathbb{H} \setminus \mathbb{R}$ with $L + M + N = 0$. On $G_n \setminus \lbrace (a,b) \rbrace$,\\ we have:
\begin{equation*}
\overline{\Delta_\tp} f_{\tp, Q} \defeq \divergence_\mathbb{G} \left( \begin{Vmatrix}
Y_1 f_{\tp, Q} + Q Y_2 f_{\tp, Q}\\
Y_2 f_{\tp, Q} - Q Y_1 f_{\tp, Q}
\end{Vmatrix}^{\tp-2}_\mathbb{G} \begin{pmatrix}
Y_1 f_{\tp, Q} + Q Y_2 f_{\tp, Q}\\
Y_2 f_{\tp, Q} - Q Y_1 f_{\tp, Q}
\end{pmatrix}  \right) = 0.
\end{equation*}
\end{theorem}

\begin{proof}
The proof of Theorem \ref{ChildersClaimB} is similar to that of Theorem \ref{ChildersClaimA} and left to the reader.
\end{proof}

We then conclude the following corollary.
\begin{corollary}\label{childerssmooth}
Let $\tp>n+2$. The function $f_{\tp,Q}$, as above, is a nontrivial smooth solution to the Dirichlet problem
\begin{eqnarray*}
\left\{\begin{array}{cc}
\overline{\Delta_{\tp}}f_{\tp, Q}(\mathbf{y})=0 & \mathbf{y} \in \mathbb{G}_n\setminus\{(a,b)\} \\
0 & \mathbf{y} = (a,b).
\end{array}\right.
\end{eqnarray*}
\end{corollary}

\section{A Generalization of the Drift Term over $\mathbb{H}$}

\subsection{Case I: $\mathbf{L + M + N \ne 0}$}\mbox{}\\

Let $Q = Li + Mj +Nk \in \mathbb{H} \setminus \mathbb{R} $ with $L + M + N \ne 0$. We consider the following parameters:
\begin{eqnarray*}
	\mu &=& \frac{\sqrt{\abs{Q^2}}}{\abs{L + M + N}}\\
	\omega &=& \frac{Q}{L + M + N}\\
    \xi &=& \sqrt{\abs{Q^2}} (L + M + N)\\
    \alpha &=&  \frac{n+2-\tp - \xi n(1- \tp)}{2(n+1)(1-\tp)}\\
 \textmd{and\ }   \beta &=&  \frac{n+2-\tp + \xi n(1- \tp)}{2(n+1)(1-\tp)}
\end{eqnarray*}
where:
\begin{equation*}
\xi \ne \pm \frac{n+2-\tp}{n(\tp-1)} .
\end{equation*}
We use these constants with the functions:
\begin{eqnarray*}
	 g(y_1, y_2) &=& \mu c(y_1-a)^{n+1} + \omega (n+1)(y_2 - b)\\
    h(y_1, y_2) &=& \mu c(y_1-a)^{n+1} - \omega (n+1)(y_2 - b)
\end{eqnarray*}
to define our main function:
\begin{equation}\label{GrushinCoreA}
f_{\tp, Q} (y_1, y_2) = g(y_1, y_2)^\alpha h(y_1, y_2)^\beta.
\end{equation}
Using Equation \ref{GrushinCoreA}, we have the following theorem.

\begin{theorem}\label{GrushinClaimA}
Let $Q = Li + Mj +Nk \in \mathbb{H} \setminus \mathbb{R}$ with $L + M + N \ne 0$. On $G_n \setminus \lbrace (a,b) \rbrace$, \\ we have:
\begin{equation*}
\mathcal{G}_{\tp, Q} \left( f_{\tp, Q} \right) \defeq \Delta_{\tp} f_{\tp, Q} + Q \left[ Y_1, Y_2 \right] \left( \norm{\nabla_0 f_{\tp, Q}}^{\tp -2}_\mathbb{G} f_{\tp, Q}  \right) = 0.
\end{equation*}
\end{theorem}

\begin{proof}
Suppressing arguments and subscripts, we compute the following:
\begin{eqnarray}
Y_1 f &=& \mu c (n+1)(y_1 - a)^n g^{\alpha-1} h^{\beta-1} (\alpha h + \beta g) \label{Y1A}\\ \nonumber
\overline{Y_1 f} &=& \mu c (n+1)(y_1 - a)^n g^{\beta -1} h^{\alpha -1} (\alpha g + \beta h)\\
Y_2 f &=& \omega c (n+1) (y_1 - a)^n g^{\alpha-1} h^{\beta-1} (\alpha h - \beta g) \label{Y2A}\\ \nonumber
\overline{Y_2 f} &=& -\omega c (n+1) (y_1 - a)^n g^{\beta -1} h^{\alpha -1} (\alpha g - \beta h)\\ \nonumber
\textmd{and\ }\norm{\nabla_0 f}^2 &=& 2 \mu^2 c^2 (n+1)^2 (y_1 - a)^{2n} g^{\alpha+ \beta -1} h^{\alpha + \beta-1} \left( \alpha^2  + \beta^2  \right).
\end{eqnarray}

Using the above, we compute:
\begin{eqnarray}
Y_1 Y_1 f &=& \mu c(n+1) (y_1 - a)^{n-1} g^{\alpha-2} h^{\beta-2} \times \nonumber\\
    &&\quad \Big( ngh (\alpha h + \beta g) + \mu c (n+1)(y_1 - a)^{n+1} \times \nonumber \\
    &&\quad \big((\alpha h + \beta g)\big((\alpha-1)h + (\beta-1)g\big)  + gh(\alpha + \beta)\big)\Big) \nonumber \\
Y_2 Y_2 f &=& -\mu^2 c^2 (n+1)^2 (y_1 - a)^{2n} g^{\alpha-2} h^{\beta-2} \times \nonumber \\
    &&\quad\big( (\alpha h -\beta g)\big( (\alpha-1) h - (\beta-1) g \big) - gh(\alpha + \beta) \big) \nonumber \\
Y_1 \norm{\nabla_0 f}^2 &=& 4 \mu^2 c^2 (n+1)^2 (y_1 - a)^{2n-1} g^{\alpha + \beta -2} h^{\alpha + \beta-2} (\alpha^2  + \beta^2) \times \label{Y1normsqA}\\ \nonumber
    &&\quad \big(  ngh  +   \mu^2 c^2 (n+1) (y_1 - a)^{2n+2}(\alpha + \beta - 1) \big) \\
Y_2 \norm{\nabla_0 f}^2 &=& -4 \omega^2 \mu^2 c^3 (n+1)^4 (y_1 - a)^{3n}  (y_2 - b) g^{\alpha + \beta -2} h^{\alpha + \beta -2} \times \label{Y2normsqA} \\ \nonumber
&&\quad  \big( \alpha^2  + \beta^2  \big)  (\alpha + \beta -1) 
\end{eqnarray}
and
\begin{eqnarray*}
\sum_{s=1}^2 Y_s \norm{\nabla_0 f}^2 (Y_s f) &=& 4 \mu^3 c^3 (n+1)^3 (y_1 - a)^{3n-1} g^{2\alpha + \beta -3} h^{\alpha + 2\beta-3} (\alpha^2  + \beta^2) \times\\
    &&\quad \big( (\alpha h + \beta g) \big(  ngh  +   \mu^2 c^2 (n+1) (y_1 - a)^{2n+2}(\alpha + \beta - 1) \big)\\
    && \quad + \omega \mu c (n+1)^2 (y_1 - a)^{n+1} (y_2 - b) (\alpha + \beta - 1)(\alpha h - \beta g) \big)\\
\norm{\nabla_0 f}^2 (Y_1Y_1 + Y_2 Y_2 f) &=& 2 \mu^3 c^3 (n+1)^3 (y_1 - a)^{3n-1} g^{2\alpha+ \beta -3} h^{\alpha + 2\beta-3} \times\\
    &&\quad  \big( \alpha^2  + \beta^2  \big) \big( ngh (\alpha h + \beta g) + 4 \mu c (n+1)(y_1 - a)^{n+1}gh \alpha \beta \big)
\end{eqnarray*}
so that
\begin{eqnarray*}
\Delta_\tp f &=& \norm{\nabla_0 f}^{\tp-4} \left( \frac{(\tp-2)}{2} \sum_{s=1}^2 Y_s\|\nabla_0 f\|^2 (Y_s f) + \norm{\nabla_0 f}^2 (Y_1 Y_1 f + Y_2 Y_2 f) \right)\\
& = & -\xi 2^{\frac{\tp-2}{2}} \mu^{\tp-1} c^{\tp-1} n^2 (n+1)^{\tp-2} (y_1 - a)^{n(\tp-1)-1} g^{\frac{\alpha \tp + \beta(\tp-2) - \tp}{2}} h^{\frac{\alpha(\tp-2) + \beta \tp - \tp}{2}} \left( \alpha^2  + \beta^2 \right)^{\frac{\tp-2}{2}}\\
    && \quad \times   \big(  \xi \mu c (y_1 - a)^{n+1}    + \omega (1-\tp) (n+1) (y_2 - b) \big).
\end{eqnarray*}
We then compute:
\begin{eqnarray*}
Q[Y_1, Y_2]\left( \norm{\nabla_0 f}^{\tp-2} f \right) &=&  Q 2^{\frac{p-2}{2}} \mu^{p-2} c^{p-1} n (n+1)^{p-2} (y_1 - a)^{n(p-1)-1} \left( \alpha^2  + \beta^2 \right)^{\frac{p-2}{2}}\\
&&\quad \times \dfrac{\partial}{\partial y_2} \left( g^{\frac{\alpha \tp + \beta(\tp-2) - (\tp-2)}{2}} h^{\frac{\alpha(\tp-2) + \beta \tp - (\tp-2)}{2}} \right)\\
&=& \xi 2^{\frac{\tp-2}{2}} \mu^{\tp-1} c^{\tp-1} n^2 (n+1)^{\tp-2} (y_1 - a)^{n(\tp-1)-1} g^{\frac{\alpha \tp + \beta(\tp-2) - \tp}{2}} h^{\frac{\alpha(\tp-2) + \beta \tp - \tp}{2}} \\
    && \quad \times \left( \alpha^2  + \beta^2 \right)^{\frac{\tp-2}{2}}  \big(  \xi \mu c (y_1 - a)^{n+1}    + \omega (1-\tp) (n+1) (y_2 - b) \big)\\
    &=& - \Delta_\tp f.
\end{eqnarray*}
\end{proof}

\subsection{Case II: $\mathbf{L + M + N = 0}$}\mbox{}\\

Let $Q = Li + Mj +Nk \in \mathbb{H} \setminus  \mathbb{R}$ with $L + M + N = 0$. We consider the following parameters:

\begin{eqnarray*}
    \xi &=& \sqrt{2\abs{LM + LN + MN}}\\
    \alpha &=&  \frac{n+2-\tp - \xi n(1- \tp)}{2(n+1)(1-\tp)}\\
 \textmd{and\ }   \beta &=&  \frac{n+2-\tp + \xi n(1- \tp)}{2(n+1)(1-\tp)}
\end{eqnarray*}
where:
\begin{equation*}
\xi \ne \pm \frac{n+2-\tp}{n(\tp-1)} .
\end{equation*}
We use these constants with the functions:
\begin{eqnarray*}
	 g(y_1, y_2) &=& \xi c(y_1-a)^{n+1} + Q (n+1)(y_2 - b)\\
    h(y_1, y_2) &=& \xi c(y_1-a)^{n+1} - Q (n+1)(y_2 - b)
\end{eqnarray*}
to define our main function:
\begin{equation}\label{GrushinCoreB}
f_{\tp, Q} (y_1, y_2) = g(y_1, y_2)^\alpha h(y_1, y_2)^\beta.
\end{equation}
Using Equation \ref{GrushinCoreB}, we have the following theorem.

\begin{theorem}\label{GrushinClaimB}
Let $Q = Li + Mj +Nk \in \mathbb{H} \setminus \mathbb{R}$ with $L + M + N = 0$. On $G_n \setminus \lbrace (a,b) \rbrace$, \\ we have:
\begin{equation*}
\mathcal{G}_{\tp, Q} \left( f_{\tp, Q} \right) \defeq \Delta_{\tp} f_{\tp, Q} + Q \left[ Y_1, Y_2 \right] \left( \norm{\nabla_0 f_{\tp, Q}}^{\tp -2}_\mathbb{G} f_{\tp, Q}  \right) = 0.
\end{equation*}
\end{theorem}

\begin{proof}
The computations proving Theorem \ref{GrushinClaimB} are similar to those of the proof of Theorem \ref{GrushinClaimA} and are left to the reader.
\end{proof}

Observing that
\begin{eqnarray*}
    \xi \ne \pm \frac{n(\tp -1)}{n+2 - \tp} & \textmd{implies} & \tp \ne \abs{\frac{\xi(n+2) + n}{n+ \xi}}_,  \abs{\frac{\xi(n+2) - n}{n- \xi}}
\end{eqnarray*}
we have immediately the following corollary.
\begin{corollary}\label{gsmooth}
Let $\tp > \max\left\lbrace \abs{\frac{\xi(n+2) + n}{n+ \xi}}_,  \abs{\frac{\xi(n+2) - n}{n- \xi}} \right\rbrace $. Then the function $f_{\tp, Q}$ of Equation \ref{GrushinCoreA} is a nontrivial smooth solution to the Dirichlet problem
\begin{eqnarray*}
\left\{\begin{array}{cc}
\mathcal{G}_{\tp, Q} \left( f_{\tp, Q}(\mathbf{y}) \right) =0 & \mathbf{y} \in \mathbb{G}_n\setminus\{(a,b)\} \\
0 & \mathbf{y} = (a,b).
\end{array}\right.
\end{eqnarray*}
\end{corollary}

\section{The Limit as $\tp \to \infty$}

\subsection{$\tp$-Laplacian Type Generalization over $\mathbb{H}$}
Recall that on $\mathbb{G}_n\setminus\{(a,b)\}$, we have
\begin{eqnarray*}
\overline{\Delta_{\tp}}f &=&  \divergence_G(\|\Upsilon\|^{\tp-2}\Upsilon) \\ 
&=&   \| \Upsilon \|^{\tp-4}\Bigg(\frac{1}{2}(\tp-2)\big(Y_1\| \Upsilon  \|^{2} \Upsilon_1+Y_2\| \Upsilon \|^{2} \Upsilon_2\big) +\| \Upsilon \|^{2}\big(Y_1 \Upsilon_1+Y_2 \Upsilon_2\big)\Bigg)  
\end{eqnarray*}
where $\Upsilon$ defined by
\begin{eqnarray*}
 \Upsilon  & \defeq & \left( \begin{array}{c}
 \Upsilon_1\\
 \Upsilon_2
\end{array} \right)
= \left( \begin{array}{c}
Y_1f+QY_2f\\
Y_2f-QY_1f
\end{array} \right)_.
\end{eqnarray*}
Formally letting $\tp \to \infty$, we obtain:
\begin{equation*}
\overline{\Delta_{\infty}}f = (Y_1\| \Upsilon  \|^{2} )\Upsilon_1+(Y_2\| \Upsilon \|^{2}) \Upsilon_2. 
\end{equation*}

\subsubsection{Case I: $L + M + N \ne 0$}\mbox{}\\

Formally letting $\tp \to \infty$ in Equation \ref{ChildersCoreA}, we obtain:
\begin{equation*}
f_{\infty,Q}(y_1,y_2) =\
g(y_1,y_2)^{\frac{1+\xi}{2n+2}}h(y_1,y_2)^{\frac{1-\xi}{2n+2}}
\end{equation*}
where we recall the functions $g(y_1, y_2)$ and $h(y_1, y_2)$ are given by:
\begin{eqnarray*}
g(y_1,y_2) & = & \mu c(y_1-a)^{n+1}+ \omega(n+1)(y_2-b)\\ 
h(y_1,y_2) & = & \mu c(y_1-a)^{n+1}- \omega(n+1)(y_2-b).
\end{eqnarray*}
We then have the following theorem.
\begin{theorem}\label{childersinftyA}
The function $f_{\infty,Q}$, as above, is a smooth solution to the Dirichlet problem
\begin{eqnarray*}
\left\{\begin{array}{cc}
\overline{\Delta_{\infty}}f_{\infty,Q}(\mathbf{y})=0 & \mathbf{y} \in \mathbb{G}_n\setminus\{(a,b)\} \\
0 & \mathbf{y} = (a,b).
\end{array}\right.
\end{eqnarray*}
\end{theorem}
\begin{proof}
We may prove this theorem by letting $\tp\to\infty$ in a prudent multiple of Equation \eqref{gpart1} and invoking continuity (cf. Corollary \ref{childerssmooth}). For completeness, though, we compute formally. We let:
\begin{equation*}
A = \frac{1+\xi}{2n+2} \quad\text{and}\quad B = \frac{1-\xi}{2n+2}
\end{equation*}
and compute:
\begin{eqnarray*}
Y_1 f &=& \mu c(n+1)(y_1 - a)^n g^{A -1} h^{B -1} (A h + B g )\\
Y_2 f &=& \omega c(n+1)(y_1-a)^n g^{A -1} h^{B -1} \left( A h - B g \right)\\
Y_1 f + Q Y_2 f &=& \mu c(n+1) (y_1 - a)^n g^{A -1} h^{B -1} \left( A h(1 - \xi) + B g (1 + \xi) \right)\\
Y_2 f - Q Y_1 f &=& \omega c(n+1)(y_1-a)^n g^{A -1} h^{B -1} \left( A h (1 - \xi) - B g (1 + \xi)  \right)\\
\norm{\Upsilon}^2 &=& 2 \mu^2 c^2 (n+1)^2 (y_1 - a)^{2n} g^{A + B -1} h^{A +B -1}  \left( A^2 (1-\xi)^2 + B^2(1 + \xi)^2  \right).
\end{eqnarray*}
We then have:
\begin{eqnarray*}
Y_1 \norm{\Upsilon}^2 &=& 2 \mu^2 c^2(1-\xi^2)^2n(n+1)^2(y_1-a)^{2n-1}(y_2-b)^2(gh)^{\frac{-1-2n}{n+1}}\\
Y_2 \norm{\Upsilon}^2 &=& 2 \omega \mu c^3(1-\xi^2)^2n(n+1)(y_1-a)^{3n}(y_2-b)(gh)^{\frac{-1-2n}{n+1}}
\end{eqnarray*}
so that:
\begin{eqnarray*}
Y_1\|\xi\|^2 \xi_1 & = &2 \mu^3 c^4(1-\xi^2)^3n(n+1)^2(y_1-a)^{4n}(y_2-b)^2(gh)^{\frac{-1-2n}{n+1}}g^{A-1}h^{B-1}\\
Y_2\|\xi\|^2 \xi_2 & = &  -2 \mu^3 c^4(1-\xi^2)^3n(n+1)^2(y_1-a)^{4n}(y_2-b)^2(gh)^{\frac{-1-2n}{n+1}} g^{A-1}h^{B-1}.
\end{eqnarray*}
The theorem follows.
\end{proof}

\subsubsection{Case II: $L + M + N = 0$}\mbox{}\\

Formally letting $\tp \to \infty$ in Equation \ref{ChildersCoreB}, we obtain:
\begin{equation*}
f_{\infty,Q}(y_1,y_2) =\
g(y_1,y_2)^{\frac{1+\xi}{2n+2}}h(y_1,y_2)^{\frac{1-\xi}{2n+2}}
\end{equation*}
where we recall the functions $g(y_1, y_2)$ and $h(y_1, y_2)$ are given by:
\begin{eqnarray*}
g(y_1,y_2) & = & \xi c(y_1-a)^{n+1}+ Q(n+1)(y_2-b)\\ 
h(y_1,y_2) & = & \xi c(y_1-a)^{n+1}- Q(n+1)(y_2-b).
\end{eqnarray*}
We then have the following theorem.
\begin{theorem}\label{childersinftyB}
The function $f_{\infty,Q}$, as above, is a smooth solution to the Dirichlet problem
\begin{eqnarray*}
\left\{\begin{array}{cc}
\overline{\Delta_{\infty}}f_{\infty,Q}(\mathbf{y})=0 & \mathbf{y} \in \mathbb{G}_n\setminus\{(a,b)\} \\
0 & \mathbf{y} = (a,b).
\end{array}\right.
\end{eqnarray*}
\end{theorem}

\begin{proof}
The proof of Theorem \ref{childersinftyB} is similar to that of Theorem \ref{childersinftyA} and omitted.
\end{proof}

\subsection{Generalization of the Drift Term over $\mathbb{H}$}

Recall that the drift $\tp$-Laplace equation in the Grushin-type planes $\mathbb{G}_n$ is given by:
\begin{equation*}
\mathcal{G}_{\tp,Q}(f) \defeq \Delta_\tp f + Q [Y_1, Y_2] \left( \norm{\nabla_0 f}_\mathbb{G}^{\tp-2} f \right) = 0 \:.
\end{equation*}
A routine expansion of the drift term yields the observation
\begin{eqnarray*}
\mathcal{G}_{\tp,Q}(f) &=& \Delta_\tp f + Q cn (y_1 -a)^{n-1} \left( \frac{\tp-2}{2} \norm{\nabla_0 f}^{\tp-4}_\mathbb{G} \left( \dfrac{\partial}{\partial y_2} \norm{\nabla_0 f}^{2}_\mathbb{G}  \right) f + \norm{\nabla_0 f}^{\tp-2}_\mathbb{G} \dfrac{\partial}{\partial y_2}f  \right)\\
&=& 0.
\end{eqnarray*}
Dividing through by $\frac{\tp-2}{2} \norm{\nabla_0 f}^{\tp-4}_\mathbb{G}$ and formally taking the limit $\tp \to \infty$, we obtain:
\begin{equation*}
\mathcal{G}_{\infty, Q}(f) = \Delta_\infty f + Q [Y_1, Y_2] \left( \norm{\nabla_0 f}_\mathbb{G}^{2} \right) f.
\end{equation*}

\subsubsection{Case I: $L + M + N \ne 0$}
Considering Equation \ref{GrushinCoreA} and formally letting $\tp \to \infty$ yields:
\begin{equation*}
f_{\infty, Q}(y_1, y_2) = g(y_1, y_2)^{\frac{1}{2(n+1)}(1-n\xi)} h(y_1, y_2)^{\frac{1}{2(n+1)}(1+n\xi)} 
\end{equation*}
where we recall the functions $g(y_1, y_2)$ and $h(y_1, y_2)$ are given by:
\begin{eqnarray*}
g(y_1, y_2) &=& \mu c(y_1 - a)^{n+1} + \omega (n+1)(y_2-b)\\
h(y_1, y_2) &=& \mu c(y_1 - a)^{n+1} - \omega (n+1)(y_2-b).
\end{eqnarray*}

We have the following theorem.
\begin{theorem}\label{inftyGrushinBBA}
The function $f_{\infty,Q}$, as above, is a smooth solution to the Dirichlet problem
\begin{eqnarray*}
\left\{\begin{array}{cc}
\mathcal{G}_{\infty,Q} f_{\infty,Q}(\mathbf{y})=0 & \mathbf{y} \in \mathbb{G}_n\setminus\{(a,b)\} \\
0 & \mathbf{y} = (a,b).
\end{array}\right.
\end{eqnarray*}
\end{theorem}
\begin{proof}
We may prove this theorem by letting $\tp\to\infty$ in Equations \eqref{Y1A}, \eqref{Y2A}, \eqref{Y1normsqA}, \eqref{Y2normsqA} and invoking continuity (cf. Corollary \ref{gsmooth}). However, for completeness we compute formally. We let:
\begin{eqnarray*}
A = \frac{1}{2(n+1)}(1-n\xi) & \textmd{and} & B = \frac{1}{2(n+1)}(1+n\xi)
\end{eqnarray*}
and, suppressing arguments and subscripts, compute:
\begin{eqnarray*}
Y_1 f &=& \mu c (n+1)(y_1 - a)^n g^{A-1} h^{B-1} (A h + B g) \\
Y_2 f &=& \omega c (n+1) (y_1 - a)^n g^{A-1} h^{B-1} (A h - B g)\\
\norm{\nabla_0 f}^2 &=& 2 \mu^2 c^2 (n+1)^2 (y_1 - a)^{2n} g^{A + B -1} h^{A + B -1} \left( A^2  + B^2  \right) \\
Y_1 \norm{\nabla_0 f}^2 &=& 4 \mu^2 c^2 (n+1)^2 (y_1 - a)^{2n-1} g^{A + B -2} h^{A + B-2} (A^2  +B^2) \times \\ 
    &&\quad \big(  ngh  +   \mu^2 c^2 (n+1) (y_1 - a)^{2n+2}(A + B - 1) \big) \\
Y_2 \norm{\nabla_0 f}^2 &=& -4 \omega^2 \mu^2 c^3 (n+1)^4 (y_1 - a)^{3n}  (y_2 - b) g^{A + B -2} h^{A + B -2} \big( A^2  + B^2  \big)  (A + B -1)
\end{eqnarray*}
so that:
\begin{eqnarray*}
\Delta_\infty f &=& Y_1 \norm{ \nabla_0 f }^2 Y_1 f  + Y_2 \norm{ \nabla_0 f }^2 Y_2 f\\
&=& 4 \mu^3 c^3 (n+1)^3 (A^2+B^2)(y_1-a)^{3n-1}g^{2A+B-3}h^{A+2B-3}\\
&&\quad \times\Big((A h+B g)\big(ngh+ \mu^2 c^2(n+1)(A +B -1)(y_1-a)^{2n+2}\big)\\
&&\quad +\: \omega \mu c(n+1)^2(y_1-a)^{n+1}(y_2-b)(A +B-1)(A h -B g)\Big)\\
&=& 4 \xi \omega \mu^3 c^3 n^2 (n+1)^3 (y_1 - a)^{3n-1} (y_2 - b) g^{2A + B -2} h^{A + 2B -2} (A^2  + B^2).
\end{eqnarray*}
We also compute:
\begin{eqnarray*}
Q [Y_1, Y_2] \left( \norm{\nabla_0 f}_\mathbb{G}^{2} \right) f &=&  Q g^A h^B \left( cn(y_1 - a)^{n-1} \dfrac{\partial}{\partial y_2}  \norm{\nabla_0 f}^2 \right)\\
&=& -4 \xi \omega  \mu^3 c^3 n^2 (n+1)^3 (y_1 - a)^{3n-1} (y_2 - b)  g^{A + B -2} h^{A + B -2}  \big( A^2  + B^2  \big)
\end{eqnarray*}
The theorem follows.
\end{proof}

\subsubsection{Case II: $L + M + N = 0$} Considering Equation \ref{GrushinCoreB} and formally letting $\tp \to \infty$ yields:
\begin{equation*}
f_{\infty, Q}(y_1, y_2) = g(y_1, y_2)^{\frac{1}{2(n+1)}(1-n\xi)} h(y_1, y_2)^{\frac{1}{2(n+1)}(1+n\xi)} 
\end{equation*}
where we recall the functions $g(y_1, y_2)$ and $h(y_1, y_2)$ are given by:
\begin{eqnarray*}
g(y_1, y_2) &=& \xi c(y_1 - a)^{n+1} + Q (n+1)(y_2-b)\\
h(y_1, y_2) &=& \xi c(y_1 - a)^{n+1} - Q (n+1)(y_2-b).
\end{eqnarray*}
We have the following theorem.
\begin{theorem}\label{inftyGrushinBBB}
The function $f_{\infty,Q}$, as above, is a smooth solution to the Dirichlet problem
\begin{eqnarray*}
\left\{\begin{array}{cc}
\mathcal{G}_{\infty,Q} f_{\infty,Q}(\mathbf{y})=0 & \mathbf{y} \in \mathbb{G}_n\setminus\{(a,b)\} \\
0 & \mathbf{y} = (a,b).
\end{array}\right.
\end{eqnarray*}
\end{theorem}
\begin{proof}
The proof of Theorem \ref{inftyGrushinBBB} is similar to that of Theorem \ref{inftyGrushinBBA} and omitted.
\end{proof}


\end{document}